\documentclass[leqno,twoside, 12pt]{amsart}
\usepackage{amssymb, latexsym, esint}
\usepackage{color}

\usepackage{amsthm}
\usepackage{amssymb,amsfonts}
\usepackage[all,arc]{xy}
\usepackage{enumerate}
\usepackage{mathrsfs}
\usepackage{amsmath}
\usepackage{verbatim}
\usepackage{hyperref}

\theoremstyle{plain}
\numberwithin{equation}{section} \numberwithin{figure}{section}

\newtheorem{theorem}{Theorem}[section]
\newtheorem{lemma}[theorem]{Lemma}
\newtheorem{proposition}[theorem]{Proposition}

\newtheorem{definition}[theorem]{Definition}
  \usepackage{epsfig} 
\usepackage{graphicx}  \usepackage{epstopdf}
\theoremstyle{definition}
\newtheorem{remark}[theorem]{Remark}
\newtheorem{example}[theorem]{Example}

\numberwithin{equation}{section}

\begin{document}

%

%

\title[Elliptic systems in Orlicz-Sobolev spaces]{Elliptic systems in Orlicz-Sobolev spaces with critical sources in bounded domains}

\maketitle

{\small
\begin{center}
{\sc Pablo Ochoa} \\
 Universidad Nacional de Cuyo-CONICET-Universidad Juan A. Maza\\
Parque Gral. San Mart\'in 5500, Mendoza, Argentina\\ pablo.ochoa@ingenieria.uncuyo.edu.ar.
\end{center}
}

\begin{abstract}
In this paper, we show the existence of non-trivial solutions  to very general  elliptic systems with  critical non-linearities  in the sense of embeddings in Orlicz-Sobolev spaces. This allows to consider  non-linearities which do not have polynomial growth. To achieve the existence, we combine a Mountain Pass Theorem without the Palais-Smale condition with the second Concentration Compactness Principle of Lions in Orlicz-Sobolev spaces. 

\end{abstract}

\section{Introduction}
In this work, we consider elliptic systems in the framework of Orlicz-Sobolev spaces with homogeneous boundary conditions:
\begin{equation}\label{main system}
\begin{cases}
-\Delta_{g_1}u=F_s(u, v) + \lambda H_s(x, u, v) \quad \text{in }\Omega\\
-\Delta_{g_2}v=F_t(u, v)+\lambda H_t(x, u, v) \quad \text{in }\Omega\\
(u, v)\neq (0, 0) \quad\text{in }\Omega\\
u=v =0 \quad \text{ on }\partial\Omega.
\end{cases}
\end{equation}Here, $\Omega$ is a bounded domain in $\mathbb{R}^N$ with smooth boundary.  We next describe briefly the data in the system \eqref{main system}. We start with the main operators. The $g_i$-Laplacian operators are given by
$$\Delta_{g_i}u := \text{div}\left(g_i(|\nabla u|)\nabla u \right), \quad i=1, 2,$$where $G'_i(t)= g_i(t)t$, being $G_i$ $N$-functions (see Section \ref{preliminares} for details). Regarding the nonlinearity $F=F(s, t)$, we will assume that it has critical growth in the sense of the embeddings between the Orlicz-Sobolev spaces $W_0^{1, G_i}(\Omega)$ into the Orlicz-Lebesgue spaces $L^{G_i^*}(\Omega)$, where $G_i^*$ is the optimal critical $N$-function in the above embedding. Regarding the nonlinearity $H=H(s, t)$, it is assumed to be subcritical  and its role is that it helps to find nontrivial solutions to \eqref{main system}. 

The prototypes for \eqref{main system} are the following elliptic systems involving $p$-Laplace operators and critical non-linearities:
\begin{equation}\label{example1}
\begin{cases}
-\Delta_{p_1}u=\mu_1|u|^{p_1^*-2}u+ c_1|v|^{\frac{p_2^*(p_1^*-1)}{p_1^*}}+d_1v|u|^{\frac{p_1^*(p_2^*-1)}{p_2^*}-2}u+\lambda H_u(u, v)\quad \text{in }\Omega\\
-\Delta_{p_2}v=\mu_2|v|^{p_2^*-2}v+ c_2|u|^{\frac{p_1^*(p_2^*-1)}{p_2^*}-2}u +d_2u|v|^{\frac{p_2^*(p_1^*-1)}{p_1^*}-2}v +\lambda H_v(u, v)\quad \text{in }\Omega \\
u=v=0 \quad \text{on }\partial \Omega,
\end{cases}
\end{equation}and
\begin{equation}\label{example2}
\begin{cases}
-\Delta_{p}u=\mu_1|u|^{q-2}u+ c_{\alpha, \beta}|u|^{\alpha-2}u|v|^\beta+\lambda H_u(u, v)\quad \text{in }\Omega\\
-\Delta_{p}v=\mu_2|v|^{q-2}v+ d_{\alpha, \beta}|v|^{\beta-2}v|u|^\alpha+\lambda H_v(u, v)\quad \text{in }\Omega \\
u=v=0 \quad \text{on }\partial \Omega,
\end{cases}
\end{equation}where $ p, p_i>1$, $p<q<p^*$,
$$p_i^*= \dfrac{Np_i}{N-p_i},$$for $i=1, 2$, and 
$\alpha, \beta>1, \quad \alpha+\beta= p^*$. The source term $H=H(u, v)$ is assumed to have subcritical growth.

  Systems like \eqref{example1} but with subcritical growth have been discussed in \cite{BF} and references therein, leaving the critical case as in \eqref{example1} as an open problem (see also \cite{Cerda}). Regarding critical growth, systems like \eqref{example2} were considered in \cite{Alves} with $p=2$ and $q=2$, in \cite{CCM} with possibly different $p$-Laplacians and in \cite{GPZ} (see also \cite{D}) with $q=p^{*}$. Resonant quasilinear elliptic systems have been treated recently in \cite{BC} and with non-linearities depending also on the gradient in \cite{Can}.    
  
  Regarding elliptic systems with more general nonlinearities, we quote the reference \cite{FMR}, where an Orlicz-space approach is considered to treat superlinear sources in elliptic systems with the standard Laplacian $\Delta$. However, their proofs seem to depend strongly in the linearity of the main operator. See also \cite{F} where a different approach is considered, showing a link with a higher order laplacian in Orlicz-Sobolev spaces.

The main result in this paper is the following. We refer the reader to Sections \ref{preliminares} and \ref{main assumptios datap} for the details in the assumptions.  
\begin{theorem}\label{main theorem}
Let $\Omega \subset \mathbb{R}^N$ be a bounded domain with the cone property. Consider the elliptic system \eqref{main system}, where the $N$-functions $G_1$ and $G_2$ satisfy \eqref{G1}, \eqref{GG}, \eqref{GGG} and \eqref{exponents}. Regarding the data $F$ and $H$, we assume that they verify the conditions $(F1)-(F3)$ and $(H1)-(H3)$, respectively. Then, there exists $\lambda_0>0$ such that for all $\lambda \geq \lambda_0$, the system \eqref{main system} has a weak solution $(u, v)\neq (0, 0)$  in $\Omega$. 
\end{theorem} 
In particular, the above theorem applies to systems of the form
\begin{equation}\label{main system2}
\begin{cases}
-\Delta_{g_1}u=F_s(v) + \lambda H_s(v) \quad \text{in }\Omega\\
-\Delta_{g_2}v=F_t(u)+\lambda H_t(u) \quad \text{in }\Omega\\
u=v =0 \quad \text{ on }\partial\Omega.
\end{cases}
\end{equation}Observe that in this case, Theorem \ref{main theorem} gives that $u\neq 0$ and $v\neq 0$ in $\Omega$.  

In Theorem \ref{main theorem} it is assumed that the subcritical term $H$ is non-negative. This condition may be relaxed if we assume a critical growth condition from below in $F$. Thus, we have the next theorem:

\begin{theorem}\label{main theorem 2}
Let $\Omega \subset \mathbb{R}^N$ be a bounded domain with the cone property. Consider the elliptic system \eqref{main system}, where the $N$-functions $G_1$ and $G_2$ satisfy \eqref{G1}, \eqref{GG}, \eqref{GGG} and \eqref{exponents}. Regarding the data $F$ and $H$, we assume that they verify the conditions $(F1)-(F2)$ and $(F3)'$  and $(H1)-(H3)$ (but $H$ not necessarily non-negative), respectively. Then, there exists $\lambda_0>0$ such that for all $\lambda \geq \lambda_0$, the system \eqref{main system} has a weak solution $(u, v) \neq (0, 0)$ in $\Omega$. 
\end{theorem} 

We point out that the above results fits into the findings of \cite{FIN} for single equations.

We conclude with the following remark. In the case of Lane-Emden  systems of the form
\begin{equation}\label{system}
\begin{cases}-\Delta u = \phi(v) \quad \text{in }\Omega\\
-\Delta v = \psi(u)\quad \text{in }\Omega\\
u=v=0\quad \text{on }\partial\Omega
\end{cases}
\end{equation}we may obtain regularity of the solutions by analysing the regularity of solutions to higher-order operators in Orlicz-Sobolev spaces. Indeed, let us suppose that $\psi$ is odd and one to one. Let $h=\psi^{-1}$, which is also odd. Then, $-\Delta v = \psi(u)$ implies
$$
u=-h(\Delta v).
$$
Plugging into the first equation of the system \eqref{system}, we deduce
$$
\phi(v)= -\Delta u = \Delta(h(\Delta v)) = \Delta(g(\Delta v)\Delta v),
$$
where
$$
g(t)=\dfrac{h(t)}{t}.
$$ 
The operator
$$\Delta_g^2 u= \Delta \left(g(\Delta u)\Delta u\right),$$is a generalization of the biharmonic Laplacian $\Delta^2$ and it was considered in \cite{OS} and in \cite{JOS}, where regularity results of Calder\'on-Zygmund type were deduced.

The paper is organized as follows. In Section \ref{preliminares}, we give the basic results and definitions concerning $N$-functions and related spaces. In Section \ref{main assumptios datap}, we provide the main assumptions on the data and we discuss some of their consequences and examples. Also, in Section \ref{preliminary results}, we state the main auxiliary results concerning the mountain pass theorem without the Palais-Smale condition and the concentration compactness principle in the framework of Orlicz-Sobolev spaces. Finally, in  Section \ref{proof of main theorem}, we explain the proof of Theorem \ref{main theorem}, which is given in several steps, and in the last Section \ref{proof second} we discuss the proof of Theorem \ref{main theorem 2}.
  
  \section{Preliminaries}\label{preliminares}

\subsection{N-functions and basic properties}
In this section, we introduce basic definitions and preliminary results related to Orlicz spaces. We start recalling the definition of an N-function.
	\begin{definition}\label{d2.1}
		A function $G \colon [0, \infty) \rightarrow \mathbb{R}$ is called an N-function if it admits the representation
		$$G(t)= \int _{0} ^{t} g(\tau)\tau d\tau,$$
		where the function $h(t)=g(t)t$ is right-continuous for $t \geq 0$,  positive for $t >0$, non-decreasing and satisfies the conditions
		$$h(0)=0, \quad h(\infty)=\lim_{t \to \infty}h(t)=\infty.$$
		\end{definition}We extend the function $G$ evenly on the whole of $\mathbb{R}$.
		
		By \cite[Chapter 1]{KR}, an N-function has also the following properties:
		\begin{enumerate}
			\item $G$ is continuous, convex, increasing for $t\geq 0$, even and $G(0) = 0$.
	\item $G$ is super-linear at zero and at infinite, that is $$\lim_{x\rightarrow 0} \dfrac{G(x)}{x}=0$$and
	$$\lim_{x\rightarrow \infty} \dfrac{G(x)}{x}=\infty.$$
	
		\end{enumerate}Indeed, the above conditions serve as an equivalent definition of N-functions.
		
		An important property for N-functions is the following:
		\begin{definition}
		
			 We say that the N-function  $G$ satisfies the $\bigtriangleup_{2}$ condition if  there exists $C > 2$ such that
			\begin{equation*}
				G(2x) \leq C G(x) \,\,\text{~~~for all~~} x \in \mathbb{R}_+.
			\end{equation*}
			\end{definition}
			
			 Examples of functions satisfying the $\bigtriangleup_{2}$ condition are:
    \begin{itemize}
        \item $G(t)= t^{q}$, $t \geq 0$, $q > 1$;
        \item $G(t)= t^{q}(|\log(t)|+1)$, $t \geq 0$, $q > 1$;
        \item $G(t)=(1+|t|)\log(1+|t|) - |t|$;
        \item $G(t)=t^{q_1}\chi_{(0, 1]}(t) + t^{q_2}\chi_{(1, \infty)}(t)$, $t \geq 0$, $q_1, q_2 > 1$.
    \end{itemize}

    
    According to \cite[Theorem 4.1, Chapter 1]{KR}, a necessary and sufficient condition for an N-function to satisfy the $\bigtriangleup_{2}$ condition is that there is $p^{+} > 1$ such that
			\begin{equation}\label{eq p mas}
		\frac{t^2g(t)}{G(t)} \leq p^{+}, ~~~~~\forall\, t>0.
	\end{equation}In particular, \eqref{eq p mas} implies that an N-function satisfying the $\bigtriangleup_{2}$ condition does not increase more rapidly than a power function (see \eqref{G product} below).
	
	Associated to $G$ is  the N-function  complementary to it which is defined as follows:
		\begin{equation}\label{Gcomp}
			\widetilde{G} (t) := \sup \left\lbrace tw-G(w) \colon w>0 \right\rbrace .
		\end{equation}We recall that the role played by $\widetilde{G}$ is the same as the conjugate exponent functions when $G(t)=t^{p}$, $p >1$.
	
The definition of the complementary function assures  the following Young-type inequality 
	\begin{equation}\label{2.5}
		at \leq G(t)+\widetilde{G} (a) \text{  for every } a,t \geq 0.
	\end{equation}
	
	We also quote the following useful lemma.
    \begin{lemma}\label{G g}
        Let $G$ be an N-function. If $G$ satisfies \eqref{eq p mas} then 
        \[
            \tilde{G}(g(t)t) \leq (p^+-1)G(t),
        \]
        where $g(t)t=G'(t)$ and $\tilde{G}$ is the complementary function of $G.$
\end{lemma}
	
	By \cite[Theorem 4.3,   Chapter 1]{KR}, a necessary and sufficient condition for the N-function $\widetilde{G} $ complementary to $G$ to satisfy the $\bigtriangleup_{2}$ condition is that there is $p^{-} > 1$ such that
			\begin{equation}
	p^{-} \leq 	\frac{t^2g(t)}{G(t)}, ~~~~~\forall\, t>0.
	\end{equation}

	From now on, we will assume that the N-function  $G(t)= \int _{0} ^{t} g(\tau)\tau d\tau$  satisfies the following growth behaviour:
	\begin{equation}\label{G1}
		1 < p^{-} \leq \frac{t^2g(t)}{G(t)} \leq p^{+} < \infty, ~~~~~\forall t>0.
	\end{equation}Hence, $G$ and $\widetilde{G} $ both satisfy the $\bigtriangleup_{2}$ condition.

	From \eqref{G1} there follows the next comparison of N-functions with power functions:
    \begin{equation}\label{G product}
        \min\left\lbrace a^{p^{-}}, a^{p^{+}} \right\rbrace G(b) \leq G(ab) \leq  \max\left\lbrace a^{p^{-}}, a^{p^{+}} \right\rbrace G(b),
    \end{equation}for all $a, b$, where we use the convention $t^{\alpha}= |t|^{\alpha-1}t$. 
    
    \medskip 

Given two $N$-functions $A$ and $B$, we say that $B$ is essentially larger than $A$, denoted by
$A \ll B,$ if for any $c > 0$,
		\begin{equation*}
			\lim_{t \rightarrow \infty} \dfrac{A(ct)}{B(t)}=0.
		\end{equation*}

%
%
%
%

\subsection*{Orlicz-Sobolev spaces} Given an N-function $G$, with $G'(t)=g(t)t$, we define the Orlicz-Lebesgue class $L^G(\Omega)$ as follows
$$
L^G(\Omega):=\left\lbrace u: \Omega \to \mathbb{R}, \int_\Omega G(u)\,dx < \infty\right\rbrace.
$$
If $G$ satisfies the $\Delta_2$ condition, then $L^G$ becomes a vector spaces, and indeed it is a Banach space with respect to the Luxemburg norm
$$
\|u\|_G:=\inf \left\lbrace \lambda > 0: \int_\Omega G\left(\dfrac{u}{\lambda}\right)\,dx \leq 1\right\rbrace.
$$
We will denote the convex modular associated to the norm by
$$
\rho_{G}(u):=\int_\Omega G(u)\,dx.
$$

We will also consider the Orlicz-Sobolev spaces
$$
W^{1, G}(\Omega):=\left\lbrace u \in L^G(\Omega), \, |\nabla u | \in L^G(\Omega)\right\rbrace.
$$
The space $W^{1, G}(\Omega)$ equipped with the norm
$$
\|u\|_{1,G}:= \|u\|_G+\|\nabla u\|_G
$$
is a Banach space.  We denote by $W_0^{1, G}(\Omega)$ the closure of $C_0^{\infty}(\Omega)$ with respect to the norm $\|u\|_{1, G}$. By Theorem \ref{compact embedding} below, an equivalent norm in $W_0^{1, G}(\Omega)$ is
$$\|u\|:=\|\nabla u\|_G.$$We will use this norm throughout the paper. 

In the case where also $\widetilde G$ satisfies the $\Delta_2$ condition, these spaces are reflexive and separable.

It will be useful to have the following relation between modulars and norms.
\begin{lemma}\label{comp norm modular}
        Let $G$ be an N-function satisfying \eqref{G1},  and let 
        $\xi^\pm\colon[0,\infty)$ $\to\mathbb{R}$ be defined as
        \[
            \xi^{-}(t):= 
            \min \big \{  t^{p^{-}}, t^{p^{+}} \big  \} ,
            \quad \text{ and }  \quad
            \xi^{+}(t):=\max \big \{  t^{p^{-}}, t^{p^{+}} \big \} . 
         \] 
         Then
   
             $$\xi^{-}(\|u\|_{G}) \leq \rho_{G}(|u|) 
                    \leq  \xi^{+}(\|u\|_{G}).$$
    \end{lemma}

In order to have the Sobolev immersions, we need to require some assumptions on $G$ that are the analogous of $p<N$ in the classical cases. We  assume that $G$ satisfies:
\begin{equation}\label{GG}
\int_0^\delta \dfrac{G^{-1}(s)}{s^{(N+1)/N}}\,ds <\infty
\end{equation}for some $\delta>0$, and

\begin{equation}\label{GGG}
\int_k^\infty \dfrac{G^{-1}(s)}{s^{(N+1)/N}}\,ds =\infty
\end{equation}for some $k>0$. 

 For a given $N$-function $G$, define the  Sobolev conjugate function $G^{*}$  of $G$ by means of
 $$G^*(t)= G\circ H^{-1}(t),$$where
$$H(t)=\left(\int_0^t\left(\dfrac{\tau}{G(\tau)} \right)^{\frac{1}{N-1}}\,d\tau \right)^{\frac{N-1}{N}}.$$ 
 
 Then $G^*$ is an $N$-function and the following embedding theorem holds (see \cite{Ci}).
 
 \begin{theorem}\label{compact embedding}
 Let $\Omega$ be any open set with finite measure. Let  $G$ be an $N$-function satisfying \eqref{GG} and \eqref{GGG} and consider $G^*$. Then, the embedding $
 W_0^{1, G}(\Omega) \hookrightarrow L^{G^*}(\Omega)$ is continuous. Moreover,  there is a constant $C>0$ such that
 $$\|u\|_{G^*}\leq C\|\nabla u\|_G$$and $L^{G^*}(\Omega)$ is the smallest Orlicz space where the above inequality is true.
 
 Finally, if $\Omega$ satisfies the cone property and if $B$ is any $N$-function increasing essentially more slowly than $G^*$, then the embedding $W^{1, G}(\Omega) \hookrightarrow L^B(\Omega)$ is compact. 
 \end{theorem}
 
 \begin{remark}
 We point out that when $\Omega$ has finite measure, assumption \eqref{GG} may always be obtained by replacing the given $N$-function with any $N$-function which is equivalent to the original near infinite and which makes the integral finite. 
 
 Also, if \eqref{GGG} does not hold, the embedding in Theorem \ref{compact embedding} is into $L^{\infty}(\Omega)$. 
 \end{remark}

 Finally, we point out that we will denote by $C$ universal positive constants, even if it changes from line to line in the calculations.

\section{Main assumptions on the data}\label{main assumptios datap}

Regarding the $N$-functions, we will assume that all $N$-functions satisfies assumptions \eqref{G1}, \eqref{GG} and \eqref{GGG}. For such a pair $G_1$ and $G_2$ of $N$-functions, we consider the elliptic system \eqref{main system} and  we let
$$E:=W_0^{1, G_1}(\Omega) \times W_0^{1, G_2}(\Omega)$$equipped with the norm
$$\|(u, v)\|:=\|u\|_{1, G_1}+\|v\|_{1, G_2},$$for any $(u, v)\in E$.

Regarding the $N$-functions $G_1$ and $G_2$, we assume that they satisfy \eqref{G1}, \eqref{GG} and \eqref{GGG}. Moreover, if $p_i^+$ and $p_i^-$ are the associated exponents for $i=1, 2$, we will assume that

\begin{equation}\label{exponents}
\max\left\lbrace p_1^+, p_2^+\right\rbrace < \min\left\lbrace (p_1^-)^*, (p_2^-)^*\right\rbrace.
\end{equation}

We assume that $F=F(s, t)\in C^1(\mathbb{R}^2\setminus \left\lbrace (0, 0)\right\rbrace)\cap C(\mathbb{R}^2)$, for simplicity that $F(0, 0)=0$, $F\geq 0,$ and the following growth conditions:

\begin{itemize}
\item[(F1)] Critical growth: there exists $C>0$ such that
$$|F_s(s, t)|\leq C\left((G_1^*)'(|s|)+(\tilde{G_1^*})^{-1}(G_2^*(|t|)) \right)$$and
$$|F_t(s, t)|\leq C\left((G_2^*)'(|t|)+(\tilde{G_2^*})^{-1}(G_1^*(|s|)) \right).$$
\item[(F2)]Ambrosetti-Rabinowitz's condition: there are $c_0, c_0'>0$ such that for all $s$ and $t$,
$$c_0F(s, t)\leq F_s(s, t)s+F_t(s, t)t\leq c'_0F(s, t).$$Moreover, assume that
\begin{equation}\label{assump c cero}
c_0>\max\left\lbrace p_1^{-}, p_2^{-}\right\rbrace
\end{equation}
\item[(F3)] Weak homogeneity: for all functions $(u, v)\in E$ with $u, v \geq R$ in a subdomain $\Omega' \subset \Omega$ and for some $R>0$, there holds
$$\int_\Omega F(tu, tv)\,dx \geq c\left(|t|^{(p_1^{-})^*}+|t|^{(p_2^{-})^*}\right),$$for all $t$ and where $c=c(\Omega', R, u, v)>0$.

\item[(F3)'] There is a constant $C>0$ such that for all $s$ and $t$,
$$F(s, t)\geq C\left(G_1^*(s)+G_2^*(t)\right).$$
\end{itemize}

Regarding $H=H(\cdot, s, t)\in C^1(\mathbb{R}^2\setminus \left\lbrace (0, 0)\right\rbrace)\cap C(\mathbb{R}^2)$,  $H(x, 0, 0)=0$ for all $x\in \Omega$, $H\geq 0$ and the following growth conditions are assumed:

\begin{itemize}
\item[(H1)] Subritical growth: there exist $C>0$ and  $N$-functions $G_i \ll B_i \ll G^*$ such that
$$|H_s(x, s, t)|\leq C\left((B_1)'(|s|)+(\tilde{B_1})^{-1}(B_2(|t|)) \right)$$and
$$|H_t(x, s, t)|\leq C\left((B_2)'(|t|)+(\tilde{B_2})^{-1}(B_1(|s|)) \right).$$Additionally, the following relation between exponents is assumed
\begin{equation}\label{exponents B}
\max\left\lbrace p_1^+, p_2^+\right\rbrace < \min\left\lbrace p_{B_1}^-, p_{B_2}^-\right\rbrace, \quad \max\left\lbrace p_{B_1}^+, p_{B_2}^+\right\rbrace < \min\left\lbrace (p_1^-)^*, (p_2^-)^* \right\rbrace.
\end{equation}
\item[(H2)]Ambrosetti-Rabinowitz's condition: there is $c_1>0$ such that for all $s$ and $t$,
$$H_s(x, s, t)s+H_t(x, s, t)t\geq c_1H(x, s, t).$$

Moreover, assume that
\begin{equation}\label{assump c 1}
c_1>\max\left\lbrace p_1^{-}, p_2^{-}\right\rbrace
\end{equation}
\item[(H3)]Positivity:  there is a set $\Omega_0\subset \Omega$ such that for all $r_0$ and $s$ and $t$ verifying $s, t \geq r_0$, there holds
$$H(x, u, v)>0, \text{ for all }x\in \Omega_0.$$ 
\end{itemize}

Before discussing some consequences of the above assumptions, we give  examples related to the sources in systems \eqref{example1} and \eqref{example2}.

\begin{example}
Consider the nonlinearity
$$F(s, t)=|s|^\alpha|t|^\beta, \quad \alpha, \beta >p, \alpha+\beta=p^*.$$
Then, taking $G_1(t)=G_2(t)=|t|^p$, we have from
$$\alpha-1 +\beta = p^*-1$$that
$$\dfrac{1}{\frac{p^*-1}{\alpha-1}}+\dfrac{1}{\frac{p^*-1}{\beta}}=1$$and so Young's inequality gives
$$|F_s(s, t)|= |\alpha|s|^{\alpha-2}s|t|^\beta|  \leq C\left(|s|^{p^*-1}+|t|^{p^{*}-1}\right).$$A similar argument applies to $F_t$ and hence $(F1)$ holds. Assumption $(F2)$ is straightforward. Finally, for $(H3)$, take $u, v\in W_0^{1, p}(\Omega)$ and suppose $u, v \geq R>0$ in a domain $\Omega'$. Then, for any $t\in \mathbb{R}$,
$$\int_\Omega |t u|^{\alpha}|t v|^{\beta}\,dx \geq t^{\alpha+\beta}R^{\alpha+\beta}|\Omega'|.$$Hence, recalling that $\alpha+\beta=p^{*}$, we prove that $F$ satisfies $(F3)$ as well.
\end{example}

\begin{example}
Consider
$$F(s, t)= |s|^{p_1^*}+|t|^{p_2^*}+ s|t|^{\frac{p_2^*(p_1^*-1)}{p_1^*}}+ t|s|^{\frac{p_1^*(p_2^*-1)}{p_2^*}}, \quad p_1, p_2 >1.$$Assume that
\begin{equation}\label{assump p}
\max\left\lbrace p_1, p_2 \right\rbrace \leq \min\left\lbrace p_1^*, p_2^*\right\rbrace.
\end{equation}Then
$$F_s(s, t)= p_1^*|s|^{p_1^*-2}s+  |t|^{\frac{p_2^*(p_1^*-1)}{p_1^*}} + \frac{p_1^*(p_2^*-1)}{p_2^*} t|s|^{\frac{p_1^*(p_2^*-1)}{p_2^*}-2}s.$$Take
$$\alpha= \dfrac{(p_1^*-1)p_2^*}{p_1^*(p_2^*-1)-p_2^*} \quad \text{and }\quad \beta = \dfrac{p_2^*(p_1^*-1)}{p_1^*}.$$
\end{example}Then
$$\dfrac{1}{\alpha}+\dfrac{1}{\beta}=1$$and so by Young's inequality
\begin{equation*}
 t|s|^{\frac{p_1^*(p_2^*-1)}{p_2^*}-2}s \leq C\left(|t|^\beta + |s|^{\alpha\frac{p_1^*(p_2^*-1)}{p_2^*}-1} \right)= C\left(|t|^{\frac{p_2^*(p_1^*-1)}{p_1^*}} +|s|^{p_1^*-1}\right)
\end{equation*}A similar estimate is obtained for $|F_t(s, t)|$. This shows that $F$ satisfies $(F1)$. 

Regarding $(F2)$, observe that
\begin{equation}\label{comb1}
\begin{split}
sF_s(s, t)+tF_t(s, t) &= p_1^*|s|^{p_1^*}+   s|t|^{\frac{p_2^*(p_1^*-1)}{p_1^*}} + \frac{p_1^*(p_2^*-1)}{p_2^*} t|s|^{\frac{p_1^*(p_2^*-1)}{p_2^*}}\\ &\quad + p_2^*|t|^{p_2^*}+ t|s|^{\frac{p_1^*(p_2^*-1)}{p_2^*}} + \frac{p_2^*(p_1^*-1)}{p_1^*}s|t|^{\frac{p_2^*(p_1^*-1)}{p_1^*}}\\ & =p_1^*|s|^{p_1^*}+  p_2^*|t|^{p_2^*} + \left(1+\frac{p_2^*(p_1^*-1)}{p_1^*} \right) s|t|^{\frac{p_2^*(p_1^*-1)}{p_1^*}}\\&\quad + \left(1+ \frac{p_1^*(p_2^*-1)}{p_2^*}\right)t|s|^{\frac{p_1^*(p_2^*-1)}{p_2^*}}.
\end{split}
\end{equation}The upper bound in $(F2)$ is obtained straightforward. For the lower bound, observe that
\begin{equation}\label{comb2}
1+ \frac{p_1^*(p_2^*-1)}{p_2^*} = 1+ p_1^*\left(1-\dfrac{1}{p_2^*}\right)>p_1^*.
\end{equation}Similarly,
 \begin{equation}\label{comb3}
1+ \frac{p_2^*(p_1^*-1)}{p_1^*} = 1+ p_2^*\left(1-\dfrac{1}{p_1^*}\right)>p_2^*.
\end{equation}Hence, combining \eqref{comb1}, \eqref{comb2} and \eqref{comb3} and appealing to \eqref{assump p}, we obtain
$$sF_s(s, t)+tF_t(s, t) > \min\left\lbrace p_1^*, p_2^*\right\rbrace F(s, t) \geq \max\left\lbrace p_1, p_2 \right\rbrace F(s, t).$$Finally $(F3)$ follows easily.

The following lemma provides a useful consequence of assumption $(F1)$.
\begin{lemma}\label{growth F}
Let $F$ satisfy $(F1)$. Then, there $C>0$ such that
$$|F(s, t)|\leq C\left(G_1^*(s)+G_2^*(t) \right).$$
\end{lemma}
\begin{proof}
Let $s, t \in \mathbb{R}$. Then
\begin{equation}
\begin{split}
|F(s, t)|&\leq \int_0^{|s|} |F_s(r, t)|\,dr + |F(0, t)| \\ & \leq  C\int_0^{|s|} \left((G^*_1)'(r) + \tilde{G_1^*}^{-1}(G^*_2(t)) \right)\,dx + |F(0, t)|\\& = C\left( G_1^*(|s|) + \tilde{G_1^*}^{-1}(G^*_2(t))|s| \right)+ |F(0, t)|.
\end{split}
\end{equation}By Young's inequality \eqref{2.5}, we obtain
$$|F(s, t)| \leq  C\left(G_1^*(s) +G^*_2(t)\right) + |F(0, t)|.$$Now, reasoning similarly and recalling that $F(0, 0)=0$, we get
$$|F(0, t)|\leq CG_2^*(t).$$This ends the proof.
\end{proof}
\begin{remark}By reasoning as in Lemma \ref{growth F}, we obtain that $H$ satisfies
\begin{equation}\label{growth H}
|H(s, t)|\leq C\left(B_1(|s|)+B_2(|t|)\right).
\end{equation}
\end{remark}

\section{Preliminary results}\label{preliminary results}

In this section, we provide the auxiliary main results that we will employ in the proof of the Theorem \ref{main theorem}. 

The associated energy functional $\Phi=\Phi_\lambda$ to the system \eqref{main system} is
$$\Phi_\lambda(u, v)=\int_\Omega G_1(|\nabla u|)\,dx + \int_\Omega G_2(|\nabla v|)\,dx -\int_\Omega F(u, v)\,dx -\lambda \int_\Omega H(x, u, v)\,dx.$$

Since our approach is variations, weak solutions of \eqref{main system} are precisely those couples $(u, v)\in E$ such that
$$\left\langle \Phi'(u, v), (\phi, \varphi)\right\rangle=0,$$for any $\phi, \varphi\in C_0^\infty(\Omega)$.

Due to the growth rates of the right hand sides of system \eqref{main system}, the functional $\Phi$
 may not satisfies the Palais-Smale condition. So, we provide  the following Mountain Pass Theorem without the Palais-Smale condition. Its proof can be found in \cite[pag. 272]{AE}. 
\begin{theorem}\label{MPL}
Let $E$ be a Banach space and $I \in C^1(E, \mathbb{R})$. Suppose that there exist a neighbourhood $U$ of $0$ in $E$ and a constant $\alpha$ satisfying the following conditions 
\begin{itemize}
\item[(i)] $I(u) \geq \alpha$ for all $u \in \partial U$,
\item[(ii)] $I(0)< \alpha$,
\item[(iii)] there exists $u_0 \notin U$ satisfying $I(u_0)< \alpha$.
\end{itemize}Let
$$\Gamma = \left\lbrace \gamma \in C([0, 1], E): \gamma(0)=0,\,\gamma(1)=u_0\right\rbrace$$and
$$c=\inf_{\gamma \in \Gamma}\max_{u \in \gamma([0, 1])}I(u) \geq \alpha.$$Then, there exists a sequence $u_k \in E$ such that
$$I(u_k)\to c \quad \text{and }\quad I'(u_k)\to 0 \,\,\text{in }E',$$as $k \to \infty$.
\end{theorem}

A sequence satisfying the conclusion of Theorem \ref{MPL} is called a Palais-Smale sequence.

 When applying Theorem \eqref{MPL} to $\Phi$ we will find a sequence $(u_k, v_k)$ which is a Palais-Smale sequence for $\Phi$ and that will be bounded.  The next  concentration compactness result obtained in \cite{FIN} will help us to deal with such a sequence (see also \cite{FS} where a  sharper result was obtain).

\begin{theorem}\label{CCP}
Let $w_k\in W_0^{1, G}(\Omega)$ be bounded and converging weakly to $w \in W_0^{1, G}(\Omega)$. Then, there exist an at most countable set $J$, a family of distinct points $\{x_j\}_{j\in J}$ and positive constants $\nu_j$ such that
$$G^*(|w_k|)\,dx \rightharpoonup \nu \text{ in }\mathcal{M}(\mathbb{R}^n),$$
$$G(|\nabla w_k|)\,dx \rightharpoonup \mu \text{ in }\mathcal{M}(\mathbb{R}^n),$$where the nonnegative Radon measures $\mu$ and $\nu$ satisfy:
$$\nu= G^*(|w|)\,dx+\sum_{j\in J}\nu_j\delta_{x_j},$$
$$\mu \geq G^*(|\nabla w|)\,dx+\sum_{j\in J}\mu_j\delta_{x_j}.$$Moreover, the coefficients $\mu_j$ and $\nu_j$ verify
$$0< \nu_j \leq \max\left\lbrace S_0^{(p^-)^*}\mu_j^{(p^-)^*/p^{-}},S_0^{(p^+)^*}\mu_j^{(p^+)^*/p^{-}}, S_0^{(p^-)^*}\mu_j^{(p^-)^*/p^{+}}, S_0^{(p^-)^*}\mu_j^{(p^+)^*/p^{+}} \right\rbrace.$$
\end{theorem}
The constant $S_0>0$ is the best constant in the embedding $W_0^{1, G}(\Omega)\hookrightarrow L^{G^*}(\Omega)$. Observe that in view of \eqref{exponents}  all the powers of $\mu_j$ are greater than 1.

We end the section with a numerical lemma.
\begin{lemma}\label{numerical lemma}
Let $a, b>0$ and $p, q \in (0, 1)$. Then
$$a^p+b^q \geq (a+b)^\alpha,$$for some $\alpha>0$ which takes some of the following values: $p$, $q$ or $p/2+q/2$.
\end{lemma}
\begin{proof}
Suppose first that $a, b <1$. Then,
$$(a+b)^{\max\left\lbrace p, q\right\rbrace} \leq a^{\max\left\lbrace p, q\right\rbrace} + b^{\max\left\lbrace p, q\right\rbrace} \leq a^p+b^q,$$which proves the lemma in this case. Next, assume that $a, b\geq 1$. Then,
$$(a+b)^{\min\left\lbrace p, q\right\rbrace} \leq a^{\min\left\lbrace p, q\right\rbrace} + b^{\min\left\lbrace p, q\right\rbrace} \leq a^p+b^q.$$Finally, suppose that $a\geq 1$ and $b <1$. Assume for simplicity that $p\geq q$. Then,
$$(a+b)^{p/2+q/2}\leq a^{p/2+q/2}+b^{p/2+q/2} \leq a^p + b^q.$$This ends the proof of the lemma.
\end{proof}

\section{Proof of Theorem \ref{main theorem}}\label{proof of main theorem}
  We divide the proof of Theorem \ref{main theorem} into several steps. 
  
  \subsection{Geometry of the energy functional and boundedness of the Palais-Smale sequence} In this part, we will prove that the energy functional $\Phi$ satisfies the geometric conditions of Theorem \ref{MPL}. We start with the next lemma.
  
  \begin{lemma}There are $\alpha>0$ and $\rho>0$ such that for all $(u, v)\in E$ with $\|(u, v)\|=\rho$, there holds
  $$\Phi(u, v)\geq \alpha.$$
  \end{lemma}
  
  \begin{proof}
  Assume that $\rho>0$ is small and take $(u,v)\in E$ with $\|(u, v)\|=\rho$. Then, by Lemma \ref{growth F}, the estimate \eqref{growth H}, and Theorem \ref{compact embedding},
  \begin{equation*}
  \begin{split}
  \Phi(u, v)&\geq \|u\|_{1, G_1}^{p_1^+}+\|v\|_{1, G_2}^{p_2^+}-\left(\|u\|_{G_1^*}^{(p_1^-)^*}+\|v\|_{G_2^*}^{(p_2^-)^*}\right)-\lambda\left(\|u\|_{B_1}^{P_{B_1}^-}+\|v\|_{B_2}^{P_{B_2}^-} \right) \\ & \geq C(\|u\|_{1, G_1}+\|v\|_{1, G_2})^{\max\left\lbrace p_1^+, p_2^+\right\rbrace} -C\left(\|u\|_{1,G_1}^{(p_1^-)^*}+\|v\|_{1, G_2}^{(p_2^-)^*}\right) -\lambda C\left(\|u\|_{1,G_1}^{P_{B_1}^-}+\|v\|_{1, G_2}^{P_{B_2}^-} \right)\\& \geq C\left(\rho^{\max\left\lbrace p_1^+, p_2^+\right\rbrace}-\rho^{(p_1^-)^*}-\rho^{(p_2^-)^*}-\rho^{p_{B_1}^-}-\rho^{p_{B_2}^-}\right) \geq \alpha,
  \end{split}
  \end{equation*}for some $\alpha>0$ taking $\rho$ small and recalling the assumptions \eqref{exponents} and \eqref{exponents B} which gives $\max\left\lbrace p_1^+, p_2^+\right\rbrace<(p_i^-)^*$ and $\max\left\lbrace p_1^+, p_2^+\right\rbrace< p_{B_i}^{-}$ for $i=1, 2$.
  \end{proof}
  
  \begin{lemma}
  There exists $(u_0, v_0)\notin B(0, \rho)$ such that
  $$\Phi(u_0, v_0)<\alpha.$$
  \end{lemma}
  
  \begin{proof}Since $F$ satisfies assumptions $(F3)$, we may choose $u_0, v_0\in C_0^{\infty}(\Omega)$ verifying $(u_0, v_0)\notin B(0, \rho)$ and $u_0, v_0\geq R>0$ in $\Omega'_0$, where $\Omega'_0$ is a domain strictly contained in $\Omega_0$ (from assumption $(F3)$) and such that for all $t$:
  $$\int_\Omega F(tu_0, tv_0)\,dx \geq c\left(|t|^{(p_1^-)^*} +  |t|^{(p_2^-)^*}\right), \quad c>0.$$Hence, for $t>0$ large and recalling assumption $(H3)$, 
  \begin{equation}\label{ineq with t}
  \Phi(tu_0,tv_0)\leq |t|^{p_1^+}\|u_0\|_{1, G_1}^{p_1^+}+ |t|^{p_2^+}\|v_0\|_{1, G_2}^{p_2^+}-c\left(|t|^{(p_1^-)^*} +  |t|^{(p_2^-)^*}\right).
  \end{equation}Since $p_i^+<(p_i^-)^*$ for $i=1, 2$, by taking $t$ large enough in \eqref{ineq with t}, we conclude the proof.
  \end{proof}

  By Theorem \ref{MPL}, there is a sequence $(u_n, v_n)\in E$ such that
  $$\Phi(u_n, v_n)\to c\in \mathbb{R}, \quad \left\langle \Phi'(u_n, v_n), (u_n, v_n)\right\rangle \to 0.$$
  
  In the next result, we will prove that the Palais-Smale sequence is bounded.
  
  \begin{lemma}\label{boundedness} The sequence $(u_n, v_n)$ is bounded in $E$.
  \end{lemma}
  \begin{proof}
  Let $\mu>0$ satisfying 
  \begin{equation}\label{chouce mu}
  \dfrac{1}{c_j}<\mu < \dfrac{1}{p_i^-}
  \end{equation}for $j=0, 1$ and $i=1, 2$. The constants $c_0$ and $c_1$ come from assumptions $(F2)$ and $(H2)$. 
  
  Firstly, observe that
  \begin{equation}\label{upper bound}
  \Phi(u_n, v_n)-\mu \left\langle \Phi'(u_n, v_n), (u_n, v_n)\right\rangle \leq C(1+\|(u_n, v_n)\|_E) = C\left[ 1+\left(\|u_n\|_{1, G_1}+\|v_n\|_{1, G_2}\right)\right],
  \end{equation}for some $C>0.$
  
  Moreover, by assumption $(H2)$ and \eqref{chouce mu},
  \begin{equation}\label{lower bound}
  \begin{split}
 & \Phi(u_n, v_n)-\mu \left\langle \Phi'(u_n, v_n), (u_n, v_n)\right\rangle \\& \quad \geq (1-p_1^{-}\mu)\int_\Omega G_1(|\nabla u_n|)\,dx+ (1-p_2^{-}\mu)\int_\Omega G_2(|\nabla v_n|)\,dx    \\& \qquad +(\mu c_0 -1)\int_{\Omega}F(u_n, v_n)\,dx  +\lambda(\mu c_1 -1)\int_{\Omega}H(x, u_n, v_n)\,dx \\& \quad \geq (1-p_1^{-}\mu)\min\left\lbrace \|u_n\|_{1, G_1}^{p_1^+},\|u_n\|_{1, G_1}^{p_1^-}  \right\rbrace + (1-p_2^{-}\mu)\min\left\lbrace \|v_n\|_{1, G_2}^{p_2^+},\|v_n\|_{1, G_2}^{p_2^-}  \right\rbrace.  
  \end{split}
  \end{equation}Since $p_i^{-}>1$, we obtain from \eqref{upper bound} and \eqref{lower bound} that the sequence $(u_n, v_n)$ is bounded in $E$.
  \end{proof}

  \subsection{Analysis of the PS-sequence via the Concentration Compactness Principle}
  
  By Lemma \ref{boundedness} and the embedding Theorem \ref{compact embedding}, there are functions $(u, v)\in E$ and measures $\mu_i$, $\nu_i$ for $i=1, 2$ such that for a subsequence,
  $$(u_n, v_n) \to (u, v) \quad \text{ in }L^{A_1}(\Omega)\times L^{A_2}(\Omega), \,A_i\ll G_i^*,$$
  $$(u_n, v_n)\to (u, v) \quad a.e. \text{ in }\Omega,$$
  $$(\nabla u_n, \nabla v_n) \rightharpoonup (\nabla u, \nabla v) \quad \text{in }L^{G_1}(\Omega)\times L^{G_2}(\Omega),$$and finally
  $$(G_1^*(u_n)dx, G^*_2(v_n)dx) \rightharpoonup (\nu_1, \nu_2) \quad \text{and}\quad (G_1(|\nabla u_n|)dx, G_2(|\nabla v_n|)dx) \rightharpoonup (\mu_1, \mu_2)$$in the sense of $wk^*$-convergence of measures. By Theorem \ref{CCP}, there are an at most countable set $J$, points $x_j^{i}$ and positive coefficients $\nu_j^{i}$ and $\mu_j^{i}$ for $i=1, 2$, such that
\begin{equation}\label{measure1}
\nu^{1}= G_1^*(u)+\sum_{j\in J}\nu_j^1\delta_{x_j^1}
\end{equation}  
  and
\begin{equation}\label{measure2}
\mu^1\geq G_1(|\nabla u|)+\sum_{j\in J}\mu_j^1\delta_{x_j^1}
\end{equation}and similarly for $v$, $\mu^2$ and $\nu^2$. 

In the next result, we will prove that $J$ is indeed finite.
\begin{proposition}
The sets $\left\lbrace x_j^i \right\rbrace_{j \in J}$ for $i=1, 2$, are finite.
\end{proposition}

\begin{proof}
 Since the argument is symmetric, we will show that the set $\left\lbrace x_j^1 \right\rbrace$ is finite. Let us take $x_j^1$ fixed. Consider a smooth function $\psi$ such that $0 \leq \psi \leq 1$,
$$\psi(x)=\begin{cases}
1, \quad |x|\leq 1\\
0, \quad |x|>2
\end{cases}.$$For $\varepsilon>0$, define
$$\phi_\varepsilon(x):= \psi\left(\dfrac{x-x_j^1}{\varepsilon} \right).$$Since
$$ \Phi'(u_n, v_n) \to 0,$$testing with $(\phi_\varepsilon u_n,\phi_\varepsilon v_n)$, we obtain
\begin{equation}\label{ineq1}
\begin{split}
&\int_\Omega g_1(|\nabla u_n|)\nabla u_n \nabla(\phi_\varepsilon u_n)\,dx +\int_\Omega g_2(|\nabla v_n|)\nabla v_n \nabla(\phi_\varepsilon v_n)\,dx \\ & \qquad = \int_\Omega F_s (u_n, v_n)u_n \phi_\varepsilon\,dx +\int_\Omega F_t (u_n, v_n)v_n \phi_\varepsilon\,dx\\&\qquad + \int_\Omega H_s (x, u_n, v_n)u_n \phi_\varepsilon\,dx+\int_\Omega H_t (x, u_n, v_n)v_n \phi_\varepsilon\,dx + o_n(1)\\&\qquad \leq C\left(\int_\Omega G_1^*(u_n)\phi_\varepsilon\,dx +\int_\Omega G_2^*(v_n)\phi_\varepsilon\,dx\right)\\& \qquad +C\left(\int_\Omega B_1(u_n)\phi_\varepsilon\,dx +\int_\Omega B_2(v_n)\phi_\varepsilon\,dx\right) + o_n(1). 
\end{split}
\end{equation}
Now, 
\begin{equation}\label{ineq2}
\begin{split}
\int_\Omega g_1(|\nabla u_n|)\nabla u_n \nabla(\phi_\varepsilon u_n)\,dx &= \int_\Omega g_1(|\nabla u_n|)\nabla u_n \nabla\phi_\varepsilon u_n\,dx + \int_\Omega g_1(|\nabla u_n|)|\nabla u_n|^2 \phi_\varepsilon\,dx \\ & \geq \int_\Omega g_1(|\nabla u_n|)\nabla u_n \nabla\phi_\varepsilon u_n\,dx + p_1^{-}\int_\Omega \phi_\varepsilon G_1(|\nabla u_n|)\,dx.
\end{split}
\end{equation}By Lemma \ref{G g} we have
$$\tilde{G_1}(g_1(|\nabla u_n|)|\nabla u_n|) \leq CG_1(|\nabla u_n|)$$and so the sequence $g_1(|\nabla u_n|)|\nabla u_n|$ is uniformly bounded in $L^{\tilde{G_1}}(\Omega)$. Thus, there is $w_1\in L^{\tilde{G_1}}(\Omega)$ such that
$$ g_1(|\nabla u_n|)\nabla u_n \rightharpoonup w_1.$$Hence, we have
\begin{equation}\label{weak1}
\int_\Omega g_1(|\nabla u_n|)\nabla u_n \nabla\phi_\varepsilon u_n\,dx \to \int_\Omega w_1 \nabla\phi_\varepsilon u\,dx.
\end{equation}This follows by adding and subtracting the term
$$g_1(|\nabla u_n|)\nabla u_n \nabla \phi_\varepsilon u$$in the difference
$$\int_\Omega  g_1(|\nabla u_n|)\nabla u_n \nabla\phi_\varepsilon u_n\,dx - \int_\Omega w_1 \nabla\phi_\varepsilon u\,dx$$and recalling the strong convergence of $u_n$ to $u$ in the support of $\phi_\varepsilon$.  

Similarly, 
\begin{equation}\label{weak2}
\int_\Omega g_2(|\nabla v_n|)\nabla v_n \nabla\phi_\varepsilon v_n\,dx \to \int_\Omega w_2 \nabla\phi_\varepsilon v\,dx,
\end{equation}for some $w_2\in L^{\tilde{G_2}}(\Omega)$. Taking $n\to \infty$ and recalling \eqref{ineq1}, \eqref{ineq2}, the convergences \eqref{weak1} and \eqref{weak2} together with \eqref{measure1} and \eqref{measure2}, we obtain
\begin{equation}\label{eq0}
\int_\Omega w_1 \nabla \phi_\varepsilon u\,dx +\int_\Omega w_2 \nabla \phi_\varepsilon v\,dx + p_1^{-}\int_\Omega \phi_\varepsilon\,d \mu^1 +  p_2^{-}\int_\Omega \phi_\varepsilon\,d \mu^2 \leq C\left(\int_\Omega \phi_\varepsilon \,d\nu^1 + \int_\Omega \phi_\varepsilon \,d\nu^2\right).
\end{equation}
We next revise the convergence of the above terms as $\varepsilon\to 0^+$. Using $(u\phi_\varepsilon, v\phi_\varepsilon)$ as test functions in $\Phi'(u_n, v_n)\to 0$, we derive
\begin{equation}\label{eq3}
\begin{split}
&\int_\Omega g_1(|\nabla u_n|)\nabla u_n \nabla (u\phi_\varepsilon)\,dx +\int_\Omega g_2(|\nabla v_n|)\nabla v_n \nabla (v\phi_\varepsilon)\,dx -\int_\Omega F_s(u_n, v_n)u \phi_\varepsilon\,dx \\ & \qquad- \int_\Omega F_t(u_n, v_n)v \phi_\varepsilon\,dx -\int_\Omega H_s(x, u_n, v_n)u \phi_\varepsilon\,dx - \int_\Omega H_t(x, u_n, v_n)v \phi_\varepsilon\,dx = o_n(1).
\end{split}
\end{equation}By assumption $(F1)$, the sequences $F_s(u_n, v_n)$ and $F_t(u_n, v_n)$ are bounded in $L^{\tilde{G_1^*}}(\Omega)$ and $L^{\tilde{G_1^*}}(\Omega)$, respectively. Hence, there are functions $b_1 \in L^{\tilde{G_1^*}}(\Omega)$ and $b_2 \in L^{\tilde{G_2^*}}(\Omega)$ such that
$$F_s(u_n, v_n) \rightharpoonup b_1 \quad \text{and }F_t(u_n, v_n) \rightharpoonup b_2.$$Similarly, there are $h_1\in L^{\tilde{B_1}}(\Omega)$ and $h_2\in L^{\tilde{B_2}}(\Omega)$ such that
$$H_s(x, u_n, v_n) \rightharpoonup h_1 \quad \text{and }H_t(x, u_n, v_n) \rightharpoonup h_2.$$

Thus, when $n\to \infty$ in \eqref{eq3}, we get
\begin{equation}\label{eq4}
\begin{split}
&\int_\Omega w_1 \nabla \phi_\varepsilon u\,dx +\int_\Omega w_2 \nabla \phi_\varepsilon v\,dx= \int_\Omega b_1 u \phi_\varepsilon\,dx + \int_\Omega b_2 v \phi_\varepsilon\,dx-\int_\Omega w_1\nabla u \phi_\varepsilon\,dx \\ & \qquad  -\int_\Omega w_2\nabla v \phi_\varepsilon\,dx+ \int_\Omega h_1 u \phi_\varepsilon\,dx + \int_\Omega h_2 v \phi_\varepsilon\,dx.
\end{split}
\end{equation}By dominated convergence theorem, the right-hand side of \eqref{eq4} converges to $0$ as $\varepsilon\to 0^+$. 

Assume first that $x_j^1 \notin \left\lbrace x_i^2\right\rbrace_{i\in J}$.
Letting $\varepsilon\to 0^+$ in \eqref{eq0}, we obtain
\begin{equation}
0<p_1^{-}\mu_j^1 \leq C \nu_j^1.
\end{equation}Hence, by Theorem \ref{CCP},
$$p_1^{-}\mu_j^1 \leq C \nu_j^1 \leq C(\mu_j^1)^\alpha,$$for some $\alpha>1$. Therefore, 
\begin{equation}\label{bound mu}
\mu_j^1>c
\end{equation}for  some constant $c>0$ independent of $j$.

If now $x^1_j=x_i^2$ for some $i$. Then, we obtain 
$$0<p_1^{-}\mu_j^1+p_2^{-}\mu_i^2 \leq C \nu_j^1+C\nu_i^2.$$We may take $\mu_i^2<<1$. Then, by Theorem \ref{CCP},
 $$0<p_1^{-}\mu_j^1+p_2^{-}\mu_i^2 \leq C \nu_j^1+C\nu_i^2\leq C (\mu_j^1)^\alpha+C (\mu_i^2)^\beta$$for some $\alpha, \beta >1$. Thus,
 $$p_1^{-}\mu_j^1 \leq C (\mu_j^1)^\alpha+C (\mu_i^2)^\beta- p_2^{-}\mu_i^2 \leq  C (\mu_j^1)^\alpha$$and we arrive at the same conclusion as before.  In this way, if $x_j^1 \notin \left\lbrace x_i^2\right\rbrace_{i\in J}$ for infinite $j's$, since 
\begin{equation}\label{series}
\sum_j \mu_j^1 <\infty,
\end{equation}we get from \eqref{bound mu} a contradiction. Then, this implies that  $x_j^1 \in \left\lbrace x_i^2\right\rbrace_{i\in J}$ holds for infinite $j's$. But the same reasoning gives a contradiction  with \eqref{series}. Therefore, the set $J$ must be finite. 
\end{proof}

\begin{lemma}\label{strong convergence}Let $K\subset \Omega\setminus \left(\left\lbrace x_j^1\right\rbrace \cup \left\lbrace x_j^2\right\rbrace \right)$ be compact. Then,
\begin{equation}\label{eq8}
u_n\to u \quad \text{and }\quad v_n\to v
\end{equation}in $L^{G_1^*}(K)$ and $L^{G_2^*}(K)$ respectively, as $n\to \infty$. 
\end{lemma}
\begin{proof}
Let $d:=\text{dist}(K, \left\lbrace x_j^1\right\rbrace \cup \left\lbrace x_j^2\right\rbrace )>0$. Fix $R>0$ so that $K\subset B_R(0)$ and for $\varepsilon>0$ consider the set
$$A_\varepsilon:=\left\lbrace x\in B_R(0): \text{dist}(x, K)<\varepsilon\right\rbrace.$$Take $\varepsilon<d$. Now, consider a smooth function $\phi \in C_0^\infty(\mathbb{R}^N)$ such that $0\leq \phi \leq 1$ and
$$\phi(x)=\begin{cases} 1, \quad x\in A_{\varepsilon/2}\\0, \quad x \in \Omega\setminus A_\varepsilon.\end{cases}.$$Observe that
$$K\subset A_{\varepsilon/2}\subset \Omega\setminus \left(\left\lbrace x_j^1\right\rbrace \cup \left\lbrace x_j^2\right\rbrace \right) \cap B_R(0),$$implies
$$\int_K G_1^*(u_n)\,dx \leq \int_{A_\varepsilon}\phi G_1^*(u_n)\,dx = \int_{\Omega}\phi G_1^*(u_n)\,dx.$$Thus, by Theorem \ref{CCP},
$$\limsup_{n\to \infty}\int_K G_1^*(u_n)\,dx \leq \int_{\Omega}\phi \,d\nu^1= \int_{\Omega}\phi G_1^*(u)\,dx =  \int_{A_\varepsilon}\phi G_1^*(u)\,dx.$$As $\varepsilon\to 0^+$, it follows that
$$\limsup_{n\to \infty}\int_K G_1^*(u_n)\,dx \leq \int_{K} G_1^*(u)\,dx$$which together with Fatou's Lemma give
$$\lim_{n\to \infty}\int_K G_1^*(u_n)\,dx = \int_{K} G_1^*(u)\,dx.$$Finally, by the Brezis-Lieb Lemma (\cite[Lemma 3]{BL}), we derive \eqref{eq8}. 
\end{proof}

\begin{lemma}\label{conv gradients}Let $K\subset \Omega\setminus \left(\left\lbrace x_j^1\right\rbrace \cup \left\lbrace x_j^2\right\rbrace \right)$ be compact. Then,
\begin{equation}\label{eq6}
\int_K \left(g_1(|\nabla u_n|)\nabla u_n-g_1(|\nabla u|)\nabla u\right)(\nabla u_n-\nabla u)\,dx\to 0 
\end{equation}and
\begin{equation}\label{eq7}
\int_K \left(g_2(|\nabla v_n|)\nabla v_n-g_2(|\nabla v|)\nabla v\right)(\nabla v_n-\nabla v)\,dx\to 0 
\end{equation}as $n\to \infty$. 
\end{lemma}

\begin{proof}

Take a smooth function $\phi \in C_0^{\infty}(\Omega)$ such that $0\leq \phi\leq 1$, $\phi = 1$ in $K$ and 
$$supp(\phi)\cap \left(\left\lbrace x_j^1\right\rbrace \cup \left\lbrace x_j^2\right\rbrace \right) = \emptyset.$$Then,
\begin{equation}\label{eq13}
\begin{split}
0 &\leq \int_K \left(g_1(|\nabla u_n|)\nabla u_n -g_1(|\nabla u|)\nabla u\right) (\nabla u_n-\nabla u)\,dx \\& \qquad  +\int_K \left(g_2(|\nabla v_n|)\nabla v_n -g_2(|\nabla v|)\nabla v\right) (\nabla v_n-\nabla v)\,dx  \\& \qquad \leq  \int_{\Omega} \left(g_1(|\nabla u_n|)\nabla u_n -g_1(|\nabla u|)\nabla u\right) (\nabla u_n-\nabla u)\phi\,dx \\& \qquad +  \int_{\Omega} \left(g_2(|\nabla v_n|)\nabla v_n -g_2(|\nabla v|)\nabla v\right) (\nabla v_n-\nabla v)\phi\,dx
\end{split}
\end{equation}The sequences $(u_n-u)\phi$ and $(v_n-v)\phi$ are bounded in $W^{1, G_1}(\Omega)$ and $W^{1, G_2}(\Omega)$, respectively, so
\begin{equation}\label{eq10}
\left\langle \Phi'(u_n, v_n), \left( (u_n-u)\phi,(v_n-v)\phi\right)\right\rangle \to 0,\end{equation}as $n\to \infty$.  Now,
\begin{equation}\label{eq11}
\begin{split}
\left\langle \Phi'(u_n, v_n), \left( (u_n-u)\phi,(v_n-v)\phi\right)\right\rangle & = \int_{\Omega}g_1(|\nabla u_n|)\nabla u_n \nabla ((u_n-u)\phi)\,dx\\ & \qquad  + \int_{\Omega}g_2(|\nabla v_n|)\nabla v_n \nabla  ((v_n-v)\phi)\,dx \\ & -\int_{\Omega}F_s(u_n, v_n) (u_n-u)\phi\,dx -\int_{\Omega}F_t(u_n, v_n) (v_n-v)\phi\,dx\\ & -\int_{\Omega}H_s(x, u_n, v_n) (u_n-u)\phi\,dx -\int_{\Omega}H_t(x, u_n, v_n) (v_n-v)\phi\,dx. 
\end{split}
\end{equation}Recall that by assumption $(F1)$, $F_s(u_n, v_n)$ and $F_t(u_n, v_n)$ are bounded in $L^{\tilde{G_1^*}}(\Omega)$ and $L^{\tilde{G_2^*}}(\Omega)$, respectively. Hence, by H\"{o}lder's inequality and Lemma \ref{strong convergence} applied in the support of $\phi$, 
\begin{equation}\label{eq20}
\bigg|\int_{\Omega}F_s(u_n, v_n) (u_n-u)\phi\,dx \bigg|, \bigg|\int_{\Omega}F_t(u_n, v_n) (v_n-v)\phi\,dx \bigg| \to 0,
\end{equation}as $n\to \infty$.
Similarly,
\begin{equation}\label{eq201}
\bigg|\int_{\Omega}H_s(x, u_n, v_n) (u_n-u)\phi\,dx \bigg|, \bigg|\int_{\Omega}H_t(x, u_n, v_n) (v_n-v)\phi\,dx \bigg| \to 0,
\end{equation}as $n\to \infty$.

 Hence, \eqref{eq10}, \eqref{eq11}, \eqref{eq20} and  \eqref{eq201} give
\begin{equation}\label{eq21}
\int_{\mathbb{R}^N}g_1(|\nabla u_n|)\nabla u_n \nabla ((u_n-u)\phi)\,dx + \int_{\mathbb{R}^N}g_2(|\nabla v_n|)\nabla v_n \nabla  ((v_n-v)\phi)\,dx \to 0,
\end{equation}as $n\to \infty$. Now, by H\"{o}lder's inequality and Lemma \ref{strong convergence} applied in the support of $\phi$ again, we get
\begin{equation}\label{eq12}
\int_{\Omega}g_1(|\nabla u_n|)\nabla u_n \nabla \phi(u_n-u)\,dx + \int_{\Omega}g_2(|\nabla v_n|)\nabla v_n \nabla \phi (v_n-v)\,dx \to 0,
\end{equation}as $n\to \infty$. Finally, by the weak convergences $(\nabla u_n, \nabla v_n)\rightharpoonup (\nabla u, \nabla v)$ in $L^{G_1}(\Omega)\times L^{G_2}(\Omega)$, we derive
\begin{equation}\label{eq14}
\int_{\Omega}g_1(|\nabla u|)\nabla u(\nabla u_n-\nabla u)\phi\,dx + \int_{\Omega}g_2(|\nabla v|)\nabla v(\nabla v_n-\nabla v)\phi\,dx\to 0,
\end{equation}as $n\to \infty$. Finally, combining \eqref{eq21}, \eqref{eq12}, \eqref{eq14} with \eqref{eq13} give \eqref{eq6} and \eqref{eq7}. 

\end{proof}

To obtain the almost everywhere convergence of the gradients, we appeal to the following weak version of  \cite[Lemma 7.1]{O} which may obtained without the condition that $t \to G(\sqrt{t})$ is convex: there exists a constant $C>0$ such that for all $a, b \in \mathbb{R}^{N}$, $a\neq b$, we have
$$\left\langle \dfrac{g(|a|)}{|a|}a-\dfrac{g(|b|)}{|b|}b, a-b \right\rangle >0 .$$
Applying this lemma with $a=\nabla u_n, \nabla v_n$ and $b=\nabla u, \nabla v$ together with Lemma \ref{conv gradients} and \cite[Lemma 6]{DM}, give the desire a.e. convergence.

\subsection{Final argument} We first show the existence of a weak solution to the system \eqref{main system}. Observe that the sequences
$$g_1(|\nabla u_n|)\nabla u_n \quad \text{and }\quad g_2(|\nabla v_n|)\nabla v_n$$are bounded in $L^{\tilde{G_1}}(\Omega)$ and $L^{\tilde{G_2}}(\Omega)$, respectively, and converge a.e. to
 $$g_1(|\nabla u|)\nabla u \quad \text{and }\quad g_2(|\nabla v|)\nabla v,$$respectively. Then,
 $$g_1(|\nabla u_n|)\nabla u_n \rightharpoonup g_1(|\nabla u|)\nabla u$$and
  $$g_2(|\nabla v_n|)\nabla v_n \rightharpoonup g_2(|\nabla v|)\nabla v.$$Similarly, the sequences
  $$F_s(u_n, v_n) \quad \text{and }\quad F_t(u_n, v_n)$$are bounded in $L^{\tilde{G^*_1}}(\Omega)$ and $L^{\tilde{G^*_2}}(\Omega)$, respectively, and converge a. e. to
  $$F_s(u, v) \quad \text{and }\quad F_t(u, v).$$Hence,
  $$F_s(u_n, v_n)\rightharpoonup F_s(u, v) \quad \text{and }\quad F_t(u_n, v_n)\rightharpoonup F_t(u, v).$$In a similar way, there holds
  $$H_s(x, u_n, v_n)\rightharpoonup H_s(x, u, v) \quad \text{and }\quad H_t(x, u_n, v_n)\rightharpoonup H_t(x, u, v).$$
  
  Therefore, from 
  $$\left\langle \Phi'(u_n, v_n), (\phi, \varphi)\right\rangle=o_n(1),\quad \phi, \varphi\in C_0^\infty(\Omega),$$and the above comments, there holds that the pair $(u, v)$ satisfies
  $$\left\langle \Phi'(u, v), (\phi, \varphi)\right\rangle=0$$and so it solves \eqref{main system}.

  Finally, we show that $(u, v)\neq (0, 0)$. We will argue by contradiction. So assume that $(u, v)=(0, 0)$.
  
  By \eqref{ineq with t}, we have for $u_0, v_0\in C_0^{\infty}(\Omega_0)$, $u_0, v_0 \geq R>0$ in $\Omega'_0$ a subdomain strictly contained in $\Omega_0$ (where $\Omega_0$ comes from assumption $(H3)$), that
  $$0 < c_\lambda \leq \max_{t\geq 0}\Phi_\lambda(tu_0, tv_0)= \Phi_\lambda(t_\lambda u_0, t_\lambda v_0),$$ where $t_\lambda\in (0, T_0)$ for $T_0>0$ independent of $\lambda$. We next show that $t_\lambda \to 0 $ as $\lambda\to \infty$. By contradiction, suppose there are a sequence $\lambda_j\to \infty$ and $\delta>0$ such that
  $$\delta\leq t_{\lambda_j}\leq T_0.$$Observe that assumption $(H3)$ implies that
  $$c=\min\left\lbrace H(x, tu_0(x), tv_0(x)): x\in \Omega'_0, \,\delta  \leq t \leq T_0  \right\rbrace>0.$$
Hence, by \eqref{ineq with t}, there is a bounded term $C(t_{\lambda_j})$ such that
$$0< c_{\lambda_j}\leq \max \Phi_{\lambda_j}(tu_0, tv_0)=\Phi_{\lambda_j}(t_{\lambda_j}u_0, t_{\lambda_j}v_0)  \leq C(t_{\lambda_j}) - c \lambda_j|\Omega'_0| \to -\infty,$$as $j\to \infty$, which is a contradiction. 

So far, we have obtained that as $\lambda\to \infty$,
\begin{equation}\label{c lamda}
c_\lambda \to 0^+.
\end{equation}
Since $(u,v)=(0, 0)$, we get
\begin{equation}\label{eq19}
\bigg|\int_\Omega H_s(x, u_n, v_n)u_n+H_t(x, u_n, v_n)v_n \,dx \bigg|\leq C\left(\|u_n\|_{B_1}\|H_s(x, u_n, v_n)\|_{\tilde{B_1}} +\|v_n\|_{B_2}\|H_t(x, u_n, v_n)\|_{\tilde{B_2}}  \right) \to 0,
\end{equation}as $n\to \infty$. Hence, from
$$\left\langle \Phi'_\lambda(u_n, v_n), (u_n, v_n)\right\rangle \to 0,$$as $n\to \infty$, we get by \eqref{eq19} that for a subsequence if necessary,
$$\lim_{n\to \infty}\int_\Omega g_1(|\nabla u_n|)|\nabla u_n|^2 +  g_2(|\nabla v_n|)|\nabla v_n|^2\,dx = \lim_{n\to \infty}\int_\Omega F_s(u_n, v_n)u_n + F_t(u_n, v_n)v_n\,dx = k.$$

Observe that $k>0$ since
$$0<c_\lambda = \lim_{n\to \infty}\Phi_\lambda(u_n, v_n) \leq \dfrac{k}{\min\left\lbrace p_1^{-},p_2^{-}\right\rbrace},$$where we have used that $F \geq 0$, \eqref{eq19},  and \eqref{G1}. 

Next,  we will obtain a uniform lower bound for $k$. Observe that by Lemma \ref{comp norm modular}
\begin{equation*}
\begin{split}
\|\nabla u_n\|_{G_1} + \|\nabla v_n\|_{G_2} & \leq \max \left\lbrace \left(\int_\Omega G_1(|\nabla u_n|)\,dx\right)^{1/p_1^+},\left(\int_\Omega G_1(|\nabla u_n|)\,dx\right)^{1/p_1^-}  \right\rbrace \\ & + \max \left\lbrace \left(\int_\Omega G_2(|\nabla v_n|)\,dx\right)^{1/p_2^+},\left(\int_\Omega G_2(|\nabla v_n|)\,dx\right)^{1/p_2^-}  \right\rbrace \\ & \leq \max \left\lbrace \left(\dfrac{1}{p_1^{-}}\int_\Omega g_1(|\nabla u_n|)|\nabla u_n|^2\,dx\right)^{1/p_1^+},\left(\dfrac{1}{p_1^{-}}\int_\Omega g_1(|\nabla u_n|)|\nabla u_n|^2\,dx\right)^{1/p_1^-}  \right\rbrace\\& + \max \left\lbrace \left(\dfrac{1}{p_2^{-}}\int_\Omega g_2(|\nabla v_n|)|\nabla v_n|^2\,dx\right)^{1/p_2^+},\left(\dfrac{1}{p_2^{-}}\int_\Omega g_2(|\nabla v_n|)|\nabla v_n|^2\,dx\right)^{1/p_2^-}  \right\rbrace \\& \leq 
\max \bigg\{ \left(\dfrac{1}{p_1^{-}}\int_\Omega g_1(|\nabla u_n|)|\nabla u_n|^2\,dx+\dfrac{1}{p_2^{-}}\int_\Omega g_2(|\nabla v_n|)|\nabla v_n|^2\,dx\right)^{1/p_1^+},\\& \left(\dfrac{1}{p_1^{-}}\int_\Omega g_1(|\nabla u_n|)|\nabla u_n|^2\,dx+\dfrac{1}{p_2^{-}}\int_\Omega g_2(|\nabla v_n|)|\nabla v_n|^2\,dx\right)^{1/p_1^-}  \bigg\}\\& +\max \bigg\{ \left(\dfrac{1}{p_1^{-}}\int_\Omega g_1(|\nabla u_n|)|\nabla u_n|^2\,dx+\dfrac{1}{p_2^{-}}\int_\Omega g_2(|\nabla v_n|)|\nabla v_n|^2\,dx\right)^{1/p_2^+},\\& \left(\dfrac{1}{p_1^{-}}\int_\Omega g_1(|\nabla u_n|)|\nabla u_n|^2\,dx+\dfrac{1}{p_2^{-}}\int_\Omega g_2(|\nabla v_n|)|\nabla v_n|^2\,dx\right)^{1/p_2^-}  \bigg\}
\end{split} 
\end{equation*}Hence,
\begin{equation}\label{im2}
\limsup_{n\to \infty}\left(\|\nabla u_n\|_{G_1} +\|\nabla v_n\|_{G_2} \right)\leq \max\left\lbrace k^{1/p_1^{+}},k^{1/p_1^{-}},k^{1/p_2^{+}},k^{1/p_2^{-}} \right\rbrace.
\end{equation}Similarly, appealing to Lemma \ref{numerical lemma},
\begin{equation*}
\begin{split}
\|u_n\|_{G^*_1}+\| v_n\|_{G^*_2} & \geq \min \left\lbrace \left(\int_\Omega G^*_1(|u_n|)\,dx\right)^{1/(p_1^+)^*},\left(\int_\Omega G^*_1(| u_n|)\,dx\right)^{1/(p_1^-)^*}  \right\rbrace \\ & +\min \left\lbrace \left(\int_\Omega G^*_2(|v_n|)\,dx\right)^{1/(p_2^+)^*},\left(\int_\Omega G^*_2(| v_n|)\,dx\right)^{1/(p_1^-)^*}  \right\rbrace\\& \geq \left(\int_\Omega G^*_1(|u_n|)\,dx+\int_\Omega G^*_2(|v_n|)\,dx \right)^{\alpha},
\end{split} 
\end{equation*}for some $\alpha >0$ depending, according to the proof of Lemma \ref{numerical lemma}, on the exponents. Hence, by Lemma \ref{growth F} and assumption $(F2)$,
\begin{equation}
\begin{split}
\|u_n\|_{G^*_1}+\| v_n\|_{G^*_2} & \geq C\left(\int_\Omega F(u_n, v_n)\,dx\right)^{\alpha}\geq C\left(\int_\Omega F_s(u_n, v_n)u_n+F_t(u_n, v_n)v_n \,dx \right)^{\alpha}.
\end{split}
\end{equation}Hence,
\begin{equation}\label{im3}
\liminf_{n\to \infty}\left(\|u_n\|_{G^*_1}+\| v_n\|_{G^*_2}\right) \geq Ck^\alpha.
\end{equation}Thus, combining \eqref{im2} and \eqref{im3} with the Orlicz-Sobolev embedding Theorem \ref{compact embedding},  we deduce
$$k^{\alpha}\leq C \max\left\lbrace k^{1/p_1^{+}},k^{1/p_1^{-}},k^{1/p_2^{+}},k^{1/p_2^{-}} \right\rbrace.$$

We point out that in view of assumption \eqref{exponents}, 
$$\alpha < \dfrac{1}{p_i^+}, $$for $i=1, 2$. Then, since $k>0$,   there is a constant $K>0$ independent of $\lambda$ such that $k\geq K$. Therefore, recalling \eqref{eq19},
\begin{equation}
\begin{split}
c_\lambda  = \lim_{n\to \infty}\Phi_\lambda (u_n, v_n) & \geq \liminf_{n\to \infty}\bigg(\dfrac{1}{p_1^{+}}\int_\Omega g_1(|\nabla u_n|)|\nabla u_n|^2\,dx + \dfrac{1}{p_2^{+}}\int_\Omega g_2(|\nabla v_n|)|\nabla v_n|^2\,dx \\ & -\dfrac{1}{c_0}\int_\Omega F_s(u_n, v_n)u_n + F_t(u_n, v_n)v_n\,dx  \\ & -\dfrac{1}{c_1}\int_\Omega H_s(x,u_n, v_n)u_n + H_t(x, u_n, v_n)v_n\,dx\bigg) \\ & \geq  \left(\dfrac{1}{\max\left\lbrace p_1^+, p_2^{+}\right\rbrace}-\dfrac{1}{c_0}\right) k \\ & \geq  \left(\dfrac{1}{\max\left\lbrace p_1^+, p_2^{+}\right\rbrace}-\dfrac{1}{c_0}\right) K.
\end{split}
\end{equation}Thus, recalling \eqref{c lamda}, for $\lambda$ large enough, we obtain a contradiction. Hence, at least one $u$ or $v$ is different from $0$. This ends the proof of the theorem.

\section{Proof of Theorem \ref{main theorem 2}}\label{proof second}

To prove Theorem \ref{main theorem 2}, we observe that the only step where we have used that $H$ is nonnegative in the previous section is in the proof of the boundedness of the Palais-Smale sequence. Specifically, in \eqref{lower bound}. To overcome this difficulty, assume $(F3)'$, and hence
\begin{equation}\label{long calc}
\begin{split}
&(\mu c_0-1)\int_\Omega F(u_n, v_n)\,dx +(\mu c_1-1)\int_\Omega H(x. u_n, v_n)\,dx \\ &\qquad \geq C(\mu c_0-1)\bigg(\min\left\lbrace \|u_n\|_{G^*_1}^{(p_1^*)^+}, \|u_n\|_{G^*_1}^{(p_1^*)^-} \right\rbrace  + \min\left\lbrace \|v_n\|_{G^*_2}^{(p_2^*)^+}, \|v_n\|_{G^*_2}^{(p_2^*)^-} \right\rbrace\bigg)\\& \qquad - C(\mu c_1-1)\bigg(\min\left\lbrace \|u_n\|_{B_1}^{p_{B_1}^+}, \|u_n\|_{B_1}^{p_{B_1}^-} \right\rbrace  + \min\left\lbrace \|v_n\|_{B_2}^{p_{B_2}^+}, \|v_n\|_{B_2}^{p_{B_2}^-} \right\rbrace\bigg)\\ & \qquad \geq C(\mu c_0-1)\bigg(\min\left\lbrace \|u_n\|_{G^*_1}^{(p_1^*)^+}, \|u_n\|_{G^*_1}^{(p_1^*)^-} \right\rbrace  + \min\left\lbrace \|v_n\|_{G^*_2}^{(p_2^*)^+}, \|v_n\|_{G^*_2}^{(p_2^*)^-} \right\rbrace\bigg)\\& \qquad - C(\mu c_1-1)\bigg(\min\left\lbrace \|u_n\|_{G_1^*}^{p_{B_1}^+}, \|u_n\|_{G_1^*}^{p_{B_1}^-} \right\rbrace  + \min\left\lbrace \|v_n\|_{G_2^*}^{p_{B_2}^+}, \|v_n\|_{G_2^*}^{p_{B_2}^-} \right\rbrace\bigg)
\end{split}
\end{equation}where we have use that $L^{G_i^*}(\Omega)\hookrightarrow L^{B_i}(\Omega)$. Now, if $\|u_n\|_{G^*_1}$ and $\|v_n\|_{G^*_2}$ remains bounded, the last terms in \eqref{long calc} is bounded and so we conclude the proof in view of \eqref{upper bound} and the remaining terms in \eqref{lower bound}. If for instance $\|u_n\|_{G^*_1}$ is unbounded, then by assumption \eqref{exponents B}, the last term in \eqref{long calc} is non negative for $n$ large enough. Hence, in this case, that term may be neglected and the proof is ended as in Lemma \ref{boundedness}.

\section*{Acknowledgements}
 P. Ochoa has been partially supported by Grant B017-UNCUYO.  P. Ochoa is a member of CONICET.

\end{document}